\documentclass[12pt,twoside]{amsart}

\usepackage[OT1]{fontenc}
\usepackage[left=3.6cm,top=3.8cm,right=3.6cm]{geometry}
\usepackage{comment}
\usepackage[utf8]{inputenc}
\usepackage{enumerate}
\usepackage{amsthm,amsmath,amssymb,mathrsfs}

\usepackage{verbatim} % comentarios
\usepackage{esint}

\numberwithin{equation}{section}

\DeclareMathOperator{\bdim}{dim_B}
\DeclareMathOperator{\lbdim}{\underline{\dim}_B}

\usepackage{xcolor}
\usepackage[pagebackref]{hyperref}
\hypersetup{
   colorlinks,
    linkcolor={red!60!black},
    citecolor={blue!60!black},
    urlcolor={blue!90!black}
}

\numberwithin{equation}{section}

\theoremstyle{plain}

\newtheorem{rmk}[equation]{Remark}

\newtheorem{thm}{Theorem}[section]
\newtheorem{lemma}[thm]{Lemma}
\newtheorem{prop}[thm]{Proposition}
\newtheorem{cor}[thm]{Corollary}

\theoremstyle{plain}

\newtheorem{defn}[thm]{Definition}

\newcommand{\e}{\varepsilon}

\newcommand{\rr}{\mathbb{R}}

\newcommand{\R}{\mathbb{R}}
\newcommand{\Z}{\mathbb{Z}}

\newcommand{\eps}{\varepsilon}
\newcommand{\norm}[1]{\left\lVert #1 \right\rVert}
\newcommand{\1}{\mathbf{1}}

\title{Maximal Operators for cube skeletons}

\author{Andrea Olivo and Pablo Shmerkin}
\address
{ Departamento de Matem\'atica, Facultad de Ciencias Exactas y Naturales, Universidad de  Buenos Aires and IMAS-CONICET, Ciudad Universitaria, Pabell\'on I (C1428EGA), Ciudad de Buenos Aires, Argentina}
\email{aolivo@dm.uba.ar}
\thanks{A.O. was partially supported by Project  CONICET-PIP 11220150100355}

\address{Department of Mathematics and Statistics, Torcuato Di Tella University, and CONICET, Buenos Aires, Argentina}
\email{pshmerkin@utdt.edu}
\urladdr{http://www.utdt.edu/profesores/pshmerkin}
\thanks{P.S. was partially supported by Projects CONICET-PIP 11220150100355 and PICT 2015-3675 (ANPCyT)}

\subjclass[2010]{Primary: 28A75, 28A80, 42B25}
\keywords{averages over skeletons, maximal functions}

\begin{document}

\maketitle

\begin{abstract}
We study discretized maximal operators associated to averaging over (neighborhoods of) squares in the plane and, more generally, $k$-skeletons in $\mathbb{R}^n$. Although these operators are known not to be bounded on any $L^p$, we obtain nearly sharp $L^p$ bounds for every small discretization scale. These results are motivated by, and partially extend, recent results of T. Keleti, D. Nagy and P. Shmerkin, and of R. Thornton, on sets that contain a scaled $k$-sekeleton of the unit cube with center in every point of $\mathbb{R}^n$.
\end{abstract}

%%%%%%%%INTRODUCCION%%%%%%%%%%%%%%%%%%%%%%%%%%

\section{Introduction and main results}

\subsection{Introduction}

One of the most basic operators in analysis is the (centered) Hardy-Littlewood maximal operator in $\rr^n$, defined as
 \[
 \mathcal{M}f(x)=  \sup\limits_{ r>0} \frac{1}{\mathcal{L}_{n}(B(x,r))}\int _ {B(x,r)} |f(y)| \, d\mathcal{L}_{n},
 \]
for all functions $f \in L^1_{\rm loc}$, where $\mathcal{L}_{n}$ is the $n$-dimensional Lebesgue measure. One of the most classical results in Analysis establishes that $\mathcal{M}$ is bounded in $L^p$ for all $p \in (1,\infty]$. It is natural to consider similar operators, and study their boundedness properties, where Lebesgue measure is replaced by some other measure. In this direction, E. Stein \cite{Stein76} (for $n\ge 3$) and J. Bourgain \cite{Bourgain86} (for $n=2$) studied the case of averages over spheres. More precisely, let $\mathcal{H}^{s}$ be the $s$-dimensional Hausdorff measure, denote the $(n-1)$-sphere with center $x$ and radius $r$ by $S^{n-1}(x,r)$ and define
\[
\mathcal{M}_{\text{sph}}f(x)=  \sup\limits_{ r>0} \frac{1}{\mathcal{H}^{n-1}(S^{n-1}(x,r))}\int _ {S^{n-1}(x,r)} |f(y)| \,d\mathcal{H}^{n-1},
\]
for $f \in C(\rr^{n})$.  Stein proved that $\mathcal{M}_{\text{sph}}$ is bounded in $L^p$ for $p > n/n-1$ if $n \geq 3$ and, ten years later, Bourgain proved the (more difficult) case $n=2$; both ranges are sharp, as seen from the indicators of small spherical caps. A simple consequence of the Stein-Bourgain Theorem is that if a set $E\subset \rr^n$ contains a sphere with center in every point  of $\rr^n$, then $\mathcal{L}_{n}(E) > 0$. In the plane, this result was proved independently by Marstrand \cite{Marstrand87}, without using maximal operators. T. Wolff \cite{Wolff00} studied a more general problem where the set of centers is an arbitrary set (instead of all of $\rr^n$). More recently, I. {\L}aba and M. Pramanik \cite{LabaPramanik11} have considered a similar problem in the line, with circles replaced by certain Cantor sets. In particular, they proved that sets of fractional dimension can differentiate some $L^p$ spaces.

A similar geometric problem, in which circles are replaced by $k$-skeletons of $n$-cubes with axes-parallel sides was recently studied by T. Keleti, D. Nagy and the second author in \cite{KNS18}, for the case  $n=2$, and by R. Thornton in \cite{Thornton17} for $n\geq 3$. In \cite{KNS18} it was shown that a set  in the plane containing a $1$-skeleton with center in every point of $[0,1]^2$ can have Lebesgue measure $0$, and it was investigated how small its fractal dimension can be for different notions of dimension. The arguments from \cite{KNS18, Thornton17} are direct and do not involve any maximal operators. The goal of this paper is to study a natural $k$-skeleton maximal operator associated to this geometric problem. As pointed out in \cite{KNS18}, the most direct generalization of the Stein-Bourgain maximal operator is not of interest, so we consider a more suitable variant that, however, is not sublinear. Also, such operator cannot be bounded from $L^p$ to $L^q$ for any finite $p$, for otherwise a set with a $k$-skeleton centered at every point would have positive measure. We study, then, discretized versions of the operator, and prove nearly sharp $L^p$ bounds for them. As a corollary, we recover one of the dimension bounds from \cite{KNS18, Thornton17}.

%%%%%%%%%%DEFINICIONES BASICAS Y NOTACION%%%%%%%%%%%%%

\subsection{Definitions and notation}

Throughout this paper, an $n$-cube will always mean an $n$-dimensional cube  with all sides parallel to the axes (unless otherwise specified), that is, a set of the form
\[
\prod_{i=1}^{n} [x_i-r,x_i+r] = x + [-r,r]^n
\]
for some $x=(x_1,\ldots,x_n)\in\R^n$, $r>0$.

The expression ${n \brack k}$ stands for all $k$-element subsets of $\{1,\ldots,n\}$. For $x \in \rr^n$ and $I \in  {n \brack k}$, $x_{I}$ is the vector in $\rr^k$ formed by taking the entries of $x$ indexed by $I$.
 The $k$-skeleton of an $n$-cube $x + [-r,r]^n$ is the set $x + \bigcup_{I \in{n \brack k}} \prod_{i=1}^{n} A_{I,i}$ where $A_{I,i}=[-r,r]$ if $i \in I$ and $\{-r,r\}$ otherwise.

We denote the cardinality of $E$ by $|E|$. If $\delta >0$, then
\[
E_\delta :=\{ x \in \rr^n : d(E, x) < \delta \}
\]
denotes the open $\delta$-neighborhood of $E$, where from now on $d$ denotes the distance induced by the infinity norm.

We denote by $S_{k}(x,r)$ the $k$-skeleton of an $n$-cube with center $x$ and side length $2r$. The faces of $S_k(x,r)$ are enumerated as $S^{j}_{k}(x,r)$, $j=1,\ldots, {n \choose k}2^{n-k}$ (we recall that if $k=0$ then faces correspond to vertices, if $k=1$ to edges, etc). It is easy to see that an $n$-cube has $N(n,k)= {n \choose k}2^{n-k}$ $k$-faces. In what follows, we denote  $N=N(n,k)$ whenever $k$ and $n$ are clear from context. We also write $S_{k,\delta}(x,r):=  (S_k(x,r))_\delta$ for simplicity. Likewise, by $S^j_{k,\delta}(x,r)$ we denote the respective $\delta$-neighborhood of $S^{j}_{k}(x,r)$. Observe that
\begin{equation} \label{eq:measure-nbhd-face}
\mathcal{L}_n(S^j_{k,\delta}(x,r))= 2^{n}(r^k\delta^{n-k} + \delta^n) \leq 2^{n+1}r^k\delta^{n-k}.
\end{equation}

We denote positive constants by $C$, indicating any parameters they may depend on by subindices. Their values may change from line to line. For example, $C_{n,k}$ denotes a positive function of $n$ and $k$.

%%%%%%%%%%%FUNCION MAXIMAL - DEFINICION %%%%%%%%%

\subsection{$k$-skeleton maximal function}

Fix $0\le k<n$ and $0<\delta <1$. A first attempt at a definition for the $k$-skeleton maximal function  might be the following:
\[
\mathbf{M}^{k}_{\delta}f(x) = \sup_{r>0} \displaystyle \frac{1}{\mathcal{L}(S_{k,\delta}(x,r))} \int _ {S_{k,\delta}(x,r)} |f(y)| \,dy,
\]
However, if we take $f$ to be the indicator of a small neighborhood of $V=[-1,2]^k \times \{0\}^{n-k}$, then we see that $\mathbf{M}^k_{\delta} f(x)\geq 1$ for all $x \in [0,1]^n$ (just choose $r$ such that the the cube centered at $x$ has one of these $k$-faces inside $V$). In other words, we obtain a constant lower bound for a function with small norm from just one $k$-face of the full skeleton. This leads to unnatural (and trivial) results.

Following \cite[\S 7]{KNS18}, we propose the following as a more interesting maximal operator:
\begin{equation}\label{unrestrictedmaximal}
\mathbf{M}^{k}_{\delta}f(x) = \sup_{r>0} \min\limits_{j=1}^{N} \displaystyle\frac{1}{\mathcal{L}_n(S_{k,\delta}^{j}(x,r))} \int _ {S_{k,\delta}^{j}(x,r)} |f(y)|\, dy,
\end{equation}
This maximal operator takes into consideration all $k$-faces of an $n$-cube, but unlike most other kinds of maximal operators, is not sub-linear. In this article we focus on the restricted version of this operator, in which the side lengths are bounded away from $0$ and $\infty$:
\begin{defn} The \textit{$k$-skeleton maximal operator} with width $\delta$ is given, for each $f \in L^{1}_{\rm loc}(\rr^n)$,
by
\begin{equation*}
M^{k}_{\delta}f(x) = \sup\limits_{1 \leq r \leq2} \min\limits_{j=1}^{N}  \displaystyle\frac{1}{\mathcal{L}_n(S_{k,\delta}^{j}(x,r))} \int _ {S_{k,\delta}^{j}(x,r)} |f(y)|\, dy.
\end{equation*}

\end{defn}

See also Section \ref{sec:extensions} for an unrestricted variant. It is easy to deduce from the Hardy-Littlewood Theorem that the $M^{k}_\delta$ are bounded in $L^p$ for $p\geq 1$ (with a bound depending on $\delta$). The problem we investigate is the rate at which $\|M^{k}_\delta\|_{L^p\to L^p}$ increases as $\delta\downarrow 0$, where as usual
\[
\norm{M^k_\delta}_{L^p\to L^q}=\sup\limits_{f\neq 0} \frac{\norm{M^k_{\delta}f}_{L^q}}{\norm{f}_{L^p}}.
\]
Note that, since $M^k_{\delta}(\lambda f)= \lambda M^k_{\delta}(f)$, for $\lambda \in \rr$, it is enough to consider functions $f$ such that $\norm{f}_{L^p}=1$. The problem of finding bounds for $\norm{M^k_\delta}_{L^p\to L^q}$ for $p\neq q$ will be addressed in a forthcoming paper.

Our main result is the following:
\begin{thm} \label{thm:main-theorem}
Given $0\le k< n$, $1\le p <\infty$ and $\eps>0$, there exist  positive constants $C'(n,k,\eps), C(n,k)$ such that
\begin{equation*}
C'(n,k,\eps)\cdot \delta^{\frac{k-n}{2np}+\eps} \leq \norm{M^{k}_{\delta}}_{L^p\to L^p} \leq C(n,k)\cdot \delta^{\frac{k-n}{2np}}.
\end{equation*}
for all $\delta\in (0,1)$.
\end{thm}

We deduce the lower bound from a construction from \cite{KNS18,Thornton17}. For the proof of the upper bound, we follow the standard strategy of reducing the problem to one of geometric intersections via a procedure that involves discretization, linearization and duality. However, the fact that we are taking the minimum over all $k$-faces in the definition of $M^{k}_{\delta}$ requires us to introduce several variants in the argument. Namely, we apply a combinatorial lemma from \cite{KNS18, Thornton17} to fix one face in each $k$-skeleton in such a a way that once we come to the problem of estimating intersections we are able to get the correct bound. This extra layer in the argument requires us to take extra care in the order in which each step is applied.

\section{Proof of Theorem \ref{thm:main-theorem}}

\subsection{The lower bound}

We start by proving the lower bound in Theorem \ref{thm:main-theorem}:
\begin{prop} \label{prop:lowerbound} For any $p\in (1,\infty)$ we have the estimate
\[
\norm{M^k_{\delta}}_{L^p\to L^{p}} \ge  C'(n,k,\eps)\cdot \delta^{\frac{k-n}{2np}+\eps}.
\]
\end{prop}

\begin{proof}
By \cite[Theorem 5.3]{Thornton17}, there exists a compact set $B \subset \rr^n$ containing the $k$-skeleton of an $n$-cube around every point of $ [0,1]^n$, such that $\bdim(B)= k + \frac{(n-k)(2n-1)}{2n}$.

We cannot directly  apply our operator to the indicator function of a neighborhood of $B$ because we do not know that the side lengths are between $1$ and $2$, but this is easily dealt with by pigeonholing and rescaling. Let $h:[0,1]^n \rightarrow (0,+\infty)$ be the function defined as $h(x)=r_x$ if $x$ is the center of $S_k(x,r_x) \subset B$
 and define the sets $H_m:=h^{-1}(2^m,2^{m+1}], m\in\Z $.
 Following the construction of $B$ (see \cite[Theorem 5.3]{Thornton17}), it is easy to see that $h$ is a measurable function. Since $[0,1]^n$ is the disjoint union of the sets $H_m$, there exists $m$ such that $\mathcal{L}(H_m)>0$; we work with this $m$ for the rest of the proof.

Consider the bi-Lipschitz function $g:\rr^n \rightarrow \rr^n$ defined as $g(x)=2^{-m}x$, then $\tilde{B}:=g(B)$ contains $k$-skeletons with center in every point of $H_m$ and side length $1<r \leq 2$. Moreover, by the bi-Lipschitz stability of box dimension (see \cite [Chapter 3.1]{Falconer03}), $\bdim\tilde{B}=\bdim{B}$.

Let $f$ be the indicator function of  $\tilde{B}_{\delta}$. Then $M^{k}_{\delta}f(x)\geq 1$ for all $x \in H_m$, so that $\mathcal{L}(H_m)^{1/p} \le \norm{M^{k}_{\delta}f}_{L^p}$. On the other hand, $\norm{f}_{L^p}=\mathcal{L}(\tilde{B}_{\delta})^{1/p}$. By the definition of box dimension (see \cite[Proposition 3.2]{Falconer03}) we have $\mathcal{L}(\tilde{B}_{\delta})\le C''_{H_m,\eps} \delta^{(n-k)/2n-p\eps}$ and this yields the claim.
\end{proof}

%%%%%%%%DISCRETIZACION Y LINEALIZACION%%%%%%%%%%

\subsection{Discretization and linearization}

In this section we introduce a suitable linearization of the operator $M^{k}_\delta$. As remarked in the introduction, this requires additional work compared to similar problems in the literature.

We denote the half-open unit cube by $Q_0$, i.e. $Q_0=[0,1)^n$. For $0<\delta<1$, we define $Q^*_0:= Q_0 \cap \delta\Z^n$. From now on we assume that $1/\delta$ is an integer. Note that it is enough to prove the upper bound in Theorem \ref{thm:main-theorem} in this case.

Our next goal is to extract, given a finite set of $k$-skeletons, one face from each skeleton in such a way that the overlaps are controlled. For this, the next result plays a crucial role.
\begin{thm}\cite[$(n,\ell)$-Dimensional Main Lemma (Theorem 2.7)]{Thornton17}
  If $A\subseteq \rr^{\ell}, X\subseteq \rr^n$ are any finite sets such that
\[
\forall x \in X \, \, \exists r \in \rr^+\, \forall I \in {n \brack \ell} \forall \sigma \in \{-1,1\}^{\ell} : x_{I} + r\sigma \in A,
\]
then $|A| \geq c_{n,\ell} |X|^{\ell(2n-1)/(2n^2)}$, where $c_{n,\ell}>0$ depends only on $n,\ell$ and not on the sets $A,X$.
\end{thm}

\begin{lemma}\label{chooseelement}
There is a constant $C_{n,k}<\infty$, depending only on $n,k$, such that the following holds. Let $\{S_k(x_i,r_i)\}_{i=1}^u$ be a finite collection of $k$-skeletons in $\rr^n$. Then it is possible to choose one $k$-face of each skeleton with the following property: If $V$ is an affine $k$-plane which is a translate of a coordinate $k$-plane, then $V$ contains at most
\[
C_{n,k}\displaystyle u^{1-\frac{(n-k)(2n-1)}{2n^2}}
\]
of the chosen $k$-faces.
\end{lemma}

\begin{proof}

Given an affine $k$-plane $V$ in $\rr^n$, let $z(V)\in\rr^{n-k}$ denote the intersection point of $V$ with the subspace orthogonal to $V$. Now consider a $k$-skeleton of an $n$-cube with center $y$ and side length $2r$. Each of its $k$-faces belongs to a $k$-plane $V$ with
\[
z(V) = y_I + r \sigma ,
\]
where $\sigma \in \{-1,1\}^{n-k}$ and $I \in {n \brack n-k}$.

Let $A$ be the set of all affine $k$-planes containing some $k$-face of some  of the $S_k(x_i,r_i)$, and write
 \begin{align*}
 A' &= \{ z(V): V\in A\} \\
 &= \left\{  (x_j)_{I} + r_j\sigma :   I \in {n \brack n-k} , \sigma \in \{-1,1\}^{n-k} , 1\le j\le u\right\}.
\end{align*}
Applying the $(n,\ell)$-Dimensional Main Lemma with $\ell=n-k$ to $A'$ and $X=\{ x_1,\ldots, x_u\}$, and using the trivial bound $|A|\ge |A'|$, we get a lower bound for $|A|$ in terms of $u$:
\begin{equation*}
|A|\geq  c_{n,k} u^{(n-k)(2n-1)/2n^2}.
\end{equation*}

Let $V_1,\ldots,V_{|A|}$ be the different $k$-planes in $A$, and let $n_j$ the number of $k$-skeletons with some face contained in $V_j$. Then
\[
\displaystyle \sum_{j=1}^{|A|} n_{j}= uN \Rightarrow \exists \, j_0  \, \text{such that} \quad \displaystyle n_{j_{0}} \leq \frac{uN}{|A|} \leq c_{n,k}^{-1} N u^{1-\frac{(n-k)(2n-1)}{2n^2}}.
\]
For this value of $j_0$, choose the face of each of the $n_{j_0}$ $k$-skeletons which is contained in $V_{j_0}$.

Now, we inductively apply the same procedure for the remaining $u_1:=u-n_{j_0}$ skeletons, and we continue for as long as $u_p> u^{1-\frac{(n-k)(2n-1)}{2n^2}}$ (here $u$ is the original number of $k$-skeletons). If $u_p \le u^{1-\frac{(n-k)(2n-1)}{2n^2}}$, then we choose an arbitrary face from each of the remaining $k$ skeletons.

\end{proof}

We will now define several convenient discretizations. Given $x \in Q_0$, we will denote by $x^*$ the center of the half-open $n$-cube with vertices in $Q_0^*$ and side length $\delta$ containing $x$. %Note that $|x-x^*|\leq \sqrt{n}\delta/2$.

\begin{defn}
Let $\Gamma$ denote the family of  all functions  $\rho: Q^*_0  \rightarrow [1,2]\cap \delta\Z$.

Enumerate $Q^*_0=\{x_1,\ldots,x_s\}$. Fix also $\rho\in\Gamma$.  For simplicity, let us write $S_{k,i}= S_k(x_i,r_i)$, where $r_i=\rho(x_i)$ and $1 \leq i \leq s$.

For this family of $k$-skeletons, we define the function $\Phi_{\rho}$:
\[
\Phi_{\rho} (S_{k}(x_i,r_i)) = \ell^j_i
\]
where $\ell^j_i$ denotes the face of $S_{k,i}$ chosen as in  Lemma \ref{chooseelement}.
\end{defn}

\begin{defn} \label{linearmaximal}  Given a function $\rho\in\Gamma$ and $0<\delta <1$, if $f \in L^{1}_{\rm loc}(\rr^n)$ we define the  \textit{$(\rho,k)$-skeleton maximal function} with width $\delta$,

\begin{center}
\begin{equation*}
{\widetilde{M}}^{k}_{\rho,\delta}f : Q_0 \rightarrow \R,
\end{equation*}
\begin{equation*}
{\widetilde{M}}^{k}_{\rho,\delta}f(x)= \displaystyle\frac{1}{\mathcal{L}(\ell_{x,\delta})} \int _ {\ell_{x,\delta}} f(y)\, dy,
\end{equation*}
\end{center}
where $\ell_{x,\delta}$ is the $\delta$-neighborhood of $\ell_{x}:=\Phi_{\rho} (S_{k}(x^*, \rho(x^*)))$.
\end{defn}

By definition,  $\widetilde{M}^k_{\rho,\delta}$ is a linear operator. We shall see that it is enough to control its behavior in order to control $M^{k}_{\delta}$ over the domain $Q_0$. Let $C Q_0$ denote the $n$-cube with the same center as $Q_0$ and side length $C$. Since $r$ is bounded by $2$, in the previous definition it is enough to consider functions $f$ supported on $7Q_0$, because if $x \in Q_0$ then $S_{k,\delta}(x,r) \subseteq 7Q_0$.

\begin{lemma} \label{lem:sup-over-discretizations} There exists a constant $C_{n,k}>0$ such that  if $0 < \delta <1$,
\[
 \norm{M^k_{\delta}}_{L^{p}\rightarrow L^{p}(Q_0)}\leq C_{n,k} \sup_{\rho\in\Gamma} \norm{\widetilde{M}^k_{\rho,3\delta}}_{L^{p}\rightarrow L^{p}(Q_0)}.
\]
\end{lemma}

\begin{proof}
Given $\eps>0$, pick $f \in L^p$ with $f\ge 0$ and $\|f\|_p=1$ such that
\[
 \norm{M^k_{\delta}}_{L^{p}\rightarrow L^{p}(Q_0)}\leq \norm{M^k_{\delta}f}_p + \eps .
\]
We claim that there exists a function $\rho\in\Gamma$ (depending on $f$) for which
\begin{equation} \label{eq:claim-discretization}
M^k_{\delta}f(x) \le 3^{n-k} \widetilde{M}^k_{\rho,3\delta} f(x) + \eps \quad \text{for all $x \in Q_0$}.
\end{equation}
Indeed, if $x \in Q_0$ it is easy to see that $S_{k,\delta}^j(x,r) \subset S_{k,2\delta}^j(x^*,r)$ for all $r\in [1,2]$ and all $1\le j\le N$ and therefore
\[
M^k_{\delta}f(x)\leq 2^{n-k} M^k_{2\delta}f(x^*).
\]

Now, given $\eps>0$ as above, for each $x^*$ there exists $r'=r'(x^*) \in [1,2]$ such that
\[
M^k_{2\delta}f(x^*) \leq \min_{j=1}^{N} \frac{1}{\mathcal{L}(S^j_{k,2\delta}(x^*,r'))} \displaystyle\int_{S^j_{k,2\delta}(x^*,r')} f(y)\, dy + \eps.
\]

Let $r=r(x^*) \in [1,2] \cap \delta\Z$ be the closest point to $r'$ (if there are two, pick the leftmost one) and define the function $x \rightarrow \rho(x^*)=r$. Since $S_{k,2\delta}^j(x^*,r')\subset S_{k,3\delta}^j(x^*,r)$ for all $j$,
\begin{eqnarray*}
M^k_{2\delta}f(x^*) &\leq&  (3/2)^{n-k} \min_{j=1}^{N} \frac{1}{\mathcal{L}(S^j_{k,3\delta}(x^*,r))}  \displaystyle\int_{S^j_{k,3\delta}(x^*,r)} f(y) \,dy + \eps/2 \\ &\leq& (3/2)^{n-k}\widetilde{M}^k_{\rho,3\delta}f(x) + \eps,
\end{eqnarray*}
We have shown that \eqref{eq:claim-discretization} holds. We conclude that
\[
\norm{M^k_{\delta}f}_{L^p(Q_0)} \leq 3^{n-k} \norm{\widetilde{M}^k_{\rho, 3\delta}f}_{L^p(Q_0)} + \eps,
\]
so that, recalling the choice of $f$ and letting $\e\to 0$, we obtain the claim.
\end{proof}

%%%%%%%%%% COTAS PARA EL OPERADOR SUBLINEAL%%%%%%

\subsection{bounds for the $(\rho,k)$-skeleton maximal operator} \label{sectionbounds}
In this section we will obtain bounds for the $(\rho,k)$-skeleton maximal operator on $L^p$ for certain values of $p$ using duality. For this, we use some ideas from \cite[Chapter 22]{Mattila15}.

The $n$ canonical vectors $e_1,\ldots, e_n$ in $\rr^n$ determine ${n \choose k}$ coordinate $k$-planes. We will denote these $k$-planes as $\pi_1,\ldots,\pi_{{n \choose k}}.$
Each $k$-face of a $k$-skeleton of an $n$-cube is contained in an affine $k$-plane which is a translate of some $\pi_{\omega}$, $1\leq \omega \leq {n \choose k} $; in this case we say that this $k$-face is parallel to $\pi_{\omega}$ (In the case $k=0$, the origin is the $0$-plane determined by the axes).

\begin{defn}
Let $\rho\in\Gamma$ and let $x_1, \ldots, x_s$ be an enumeration of $Q^*_0$. Consider the $k$-skeletons $S_{k,i}=S_k(x_i, \rho(x_i))$, $i=1,\ldots,s$. We define the sets
\[
E_{\pi_\omega}:= \{ x_i \in Q^*_0 :  \Phi(S_{k,i})= \ell^i_k \,  \text{is parallel to}\, \pi_\omega \},
\]

In the case $k=0$, we just consider the whole space $Q^*_0$.

\end{defn}

For notational convenience, we write $\psi:Q_0\to Q^*_0, x\mapsto x^*$. Observe that $\{ \psi^{-1}(E_{\pi_{\omega}})\}$ is a Borel partition of $Q_0$.

\begin{prop}
 \label{pq} Let $1 < p < \infty$, $q=\frac{p}{p-1}$, $0 < \delta < 1$, $0< K < \infty$,  and let $\rho\in\Gamma$ be given. Fix a $k$-plane $\pi$  determined by the canonical vectors, let $u=|E_{\pi}|$ and enumerate $E_\pi=\{x_1,\ldots,x_u\}$. Write $\ell^i_{k,\delta}$ for the $\delta$-neighborhood of $\Phi(S_{k}(x_i, \rho(x_i)))$. Assume that
\[
 \norm{\sum_{i=1}^{u}  t_{i} {\1}_{\ell^{i}_{k,\delta}}}_{L^{q}(7Q_0)} \leq K
\]
whenever $t_1,\ldots,t_u$ are positive numbers with
\[
 \delta^{n-kq}\displaystyle\sum_{i=1}^{u} t_i^q \leq 1.
\]

Then
 \[
 \norm{\widetilde{M}^k_{\rho,\delta}f}_{L^{p}(\psi^{-1}(E_{\pi}))} \leq K \norm{f}_{L^{p}(Q_0)}\, \text{for all $f \in L^p(7Q_0)$}.
 \]

\end{prop}

\begin{proof}
Fix $f \in L^p(7Q_0)$. To begin, we estimate
\begin{align*}
\int_{\psi^{-1}(E_{\pi})} |\widetilde{M}^k_{\rho,\delta} f(x)|^{p} \, dx
&= \sum_{i=1}^{u} \int_{\psi^{-1}(x_i)}
 |\widetilde{M}^k_{\rho,\delta} f(x)|^{p} \, dx \\ &\leq \sum_{i=1}^{u} \mathcal{L}_n(\psi^{-1}(x_i)) |\widetilde{M}^k_{\rho,\delta} f(x_i)|^{p}  \\
 &= \delta^n \sum_{i=1}^{u}   |\widetilde{M}^{k}_{\rho,\delta}f(x_i)|^{p}.
\end{align*}
By duality, for any $a_{i} \geq 0$, $i=1, \ldots, u$,
\[
\displaystyle\left( \sum_{i=1}^{u} {a_{i}}^{p} \right)^{1/p}= \max  \left\{ \displaystyle\sum_{i=1}^{u} a_{i}b_{i} : b_{i}\geq 0, \sum_{i=1}^{u}{b_{i}}^{q}\le 1 \right\}.
\]
Applying this to $a_{i}= |\widetilde{M}^k_{\rho,\delta} f(x_{i})| \delta^{n/p} $, we get
\begin{align*}
 \norm{\widetilde{M}^k_{\rho,\delta}f}_{L^{p}(\psi^{-1}(E_{\pi}))}  &\leq \left( \sum_{i=1}^{u} \left(|\widetilde{M}^k_{\rho,\delta} f(x_{i})|\delta^{n/p}\right)^{p}\right)^{1/p}  \\
&=\max \left\{ \sum_{i=1}^{u} \delta^{n/p} b_i |\widetilde{M}^k_{\rho,\delta} f(x_{i})| : b_i \ge 0, \sum_{i=1}^u b_i^q\le 1 \right\} \\
&= \delta^{n-k} \max \left\{ \sum_{i=1}^{u} t_{i}  |\widetilde{M}^k_{\rho,\delta} f(x_{i})|:  \delta^{n-kq}\displaystyle\sum_{i=1}^{u} t_i^q \leq 1 \right\},
\end{align*}
making the change of variable $t_{i}=\delta^{\frac{n}{p} -(n-k) }b_i$.

Therefore there exist $t_i\ge 0$ with $\delta^{n-kq}\sum_{i=1}^{u} t_i^q \leq 1$ such that
\begin{align*}
 \norm{\widetilde{M}^k_{\rho,\delta}f}_{L^{p}(\psi^{-1}(E_{\pi}))}  &\leq \delta^{n-k}\displaystyle\sum_{i=1}^{u} t_{i} | \widetilde{M}^k_{\rho,\delta} f(x_{i})|\\
&\leq \delta^{n-k}\displaystyle\sum_{i=1}^{u} t_{i} \displaystyle\frac{1}{\mathcal{L}(\ell^{i}_{k,\delta})}\displaystyle \int _ {\ell^{i}_{k,\delta}} |f(y)|\,dy  \\
&\leq 2^{-n}\displaystyle {\sum_{i=1}^{u} t_{i}}\displaystyle \int _ {\ell^{i}_{k,\delta}} |f(y)|\,dy,
\end{align*}
using \eqref{eq:measure-nbhd-face} and that $\rho(x_i)\ge 1$ for all $i$.

Finally, by H\"{o}lder's inequality (and bounding $2^{-n}\le 1$)
\begin{align*}
\norm{\widetilde{M}^k_{\rho,\delta}f}_{L^{p}(\psi^{-1}(E_\pi))} &\leq   \displaystyle \int_{7Q_0} \left(\sum_{i=1}^{u} t_{i} {\1}_{{\ell^{i}_{k,4\delta}}}\right) |f(y)|\, dy\\
&\leq  \displaystyle \norm {\sum_{i=1}^{u} t_{i} {{\1}_{{\ell^{i}_{k,\delta}}}}}_{L^q} \norm{f}_{L^p(7 Q_0)}\\
 &\leq  K \norm{f}_{L^{p}(7Q_0)}.
 \end{align*}
This finishes the proof.
\end{proof}

Let $m'=\frac{m}{m-1}$, with $2\leq m \in \mathbb{N}$.
\begin{prop}\label{prop:q=2m} For all $0<\delta<1$, $\rho\in\Gamma$, $2 \leq m \in \mathbb{N}$ and $f \in L^{m'}(7Q_0)$,
\begin{equation} \label{p,q=2m}
 \norm{\widetilde{M}^k_{\rho,\delta}f}_{L^{m'}(Q_0)} \leq  C_{n,k} \delta^{{\frac{k-n}{2n}}.\frac{1}{m'}} \norm{f}_{L^{m'}(7Q_0)},
\end{equation}
where $C_{n,k}>0$ depends on $n,k$ only (in particular, it is independent of $\rho$ and $m$).
\end{prop}

\begin{proof}
In the course of the proof $C_{n,k}$ denote positive constants that depend on $n$ and $k$ only; their value can change from line to line.

Let $f \in L^{m'}(7Q_0)$ and consider the coordinate $k$-planes $\pi_1,\ldots,\pi_{{n \choose k}}$. It is enough to bound the $L^{m'}$-norm of $\widetilde{M}^k_{\rho,\delta}f$ over each set $\psi^{-1}(E_{\pi_\omega})$.  Hence we fix $\omega$ and we work with $\pi_\omega$ for the rest of the proof. Let $x_1,\ldots,x_u$ be the points in $\psi^{-1}(E_{\pi_{\omega}})$ and let $\ell^1_{k,\delta}\ldots,\ell^u_{k,\delta}$ be as in Proposition \ref{pq}. By this proposition, it is enough to show that
\begin{equation} \label{eq:Lm-norm-enough-to-show}
 \norm{\sum_{i=1}^{u}  t_{i} {\1}_{\ell^i_{k,\delta}}}_{L^{m}(7Q_0)} \leq C_{n,k} \delta^{{\frac{k-n}{2n}}.\frac{1}{m'}},
\end{equation}
whenever $t_1,\ldots,t_u$ are positive real numbers with $\delta^{n-mk}\displaystyle\sum_{i=1}^{u} t_i^m \leq 1$. Fix, then $t_1,\ldots,t_u$ satisfying this. Let
\[
\text{I} := \displaystyle \norm {\sum_{i=1}^{u} t_{i} {\1}_{{\ell^{i}_{k,\delta}}}}_{L^{m}(7Q_0)}^{m} = \sum_{i_1,\ldots ,i_{m}=1}^{u} t_{i_1}\ldots t_{i_{m}} \mathcal{L}_n({\ell^{i_1}_{k,\delta}} \cap \ldots \cap  {\ell^{i_{m}}_{k,\delta}}).
\]
The sum in $\text{I}$ involves the measures of the intersections of $\delta$-neighborhoods of parallel $k$-faces taken from $m$ different $k$-skeletons. We note that this measure will often be $0$ and, by \eqref{eq:measure-nbhd-face}, is always bounded above by $2^{n+k+1}\delta^{n-k}$.

For simplicity, let us write
\[
L_{i_1\ldots i_m} = \mathcal{L}_n\left({\ell^{i_1}_{k,\delta}} \cap \ldots \cap  \ell^{i_{m}}_{k,\delta}\right).
\]
Using H\"{o}lder's inequality, we estimate
\begin{align}
\text{I} &= \sum^{u}_{i_1,\ldots,i_{m} = 1} t_{i_1}\cdots t_{i_{m}} L_{i_1\ldots i_m}^{1/m}\cdots L_{i_1\ldots i_m}^{1/m} \nonumber \\
&\leq \left( \sum^{u}_{i_1,\ldots,i_{m} = 1} t^{m}_{i_1}L_{i_1\ldots i_m}\right)^{1/m}\ldots \left( \sum^{u}_{i_1,\ldots,i_{m} = 1} t^{m}_{i_{m}} L_{i_1\ldots i_m}\right)^{1/m} \nonumber \\
&= \sum_{i_1=1}^u t_{i_1}^m  \sum_{i_2,\ldots,i_m=1}^u L_{i_1\ldots i_m}, \label{eq:bound-I}
\end{align}
since all the factors in the second line are equal.

Recall that the faces were selected according to Lemma \ref{chooseelement}, and that any two faces which are not contained in the same plane are $\delta$-separated. It follows that
if we fix the value of $i_1$, there are at most
\[
C_{n,k}^{m-1} u^{(m-1)\left(1-\tfrac{(n-k)(2n-1)}{2n^2}\right)}
\]
tuples $(i_2,\ldots,i_m)$ such that $\ell^{i_1}_{k,\delta} \cap \ldots \cap  \ell^{i_{m}}_{k,\delta}\neq\varnothing$. Since $u=|E_\pi|\le \delta^{-n}$, for each fixed $i_1$ we estimate
\[
\sum_{i_2,\ldots,i_m=1}^{u} L_{i_1\ldots i_m} \leq 2^{n+k+1}\delta^{n-k} C_{n,k}^{m-1}\left(\delta^{-n + \frac{(n-k)(2n-1)}{2n}}\right)^{m-1}.
\]
Combining this with \eqref{eq:bound-I}, and then using that $\sum_{i=1}^u t_i^m \le \delta^{mk-n}$, after some algebra we get
\begin{align*}
\text{I} &\le  C_{n,k}^m\left(\delta^{-n + \frac{(n-k)(2n-1)}{2n}}\right)^{m-1} \sum_{i_1=1}^u t_{i_1}^m  \\
   &\le  C_{n,k}^m \delta^{\frac{(k-n)(m-1)}{2n}}.
\end{align*}
This establishes \eqref{eq:Lm-norm-enough-to-show} and finishes the proof.

\end{proof}

By letting $m\to\infty$ in Proposition \ref{prop:q=2m}, we obtain the following corollary:
\begin{cor} \label{cor:p=1} For all $0<\delta<1$, $\rho\in\Gamma$ and $f \in L^1(7Q_0)$,
\[
 \norm{\widetilde{M}^k_{\rho,\delta}f}_{L^{1}(Q_0)} \leq C_{n,k}\delta^{\frac{k-n}{2n}} \norm{f}_{L^{1}(7Q_0)}.
\]
\end{cor}

\begin{proof} Fix $f \in L^\infty(7Q_0)$. We have $\norm{f}_{m'} \rightarrow \norm{f}_{1}$ when $m \rightarrow \infty$. Likewise, since $\widetilde{M}^k_{\rho,\delta} \,f$ is bounded, we have $\norm{\widetilde{M}^k_{\rho,\delta}\,f}_{m'} \rightarrow \norm{\widetilde{M}^k_{\rho,\delta}\,f}_{1}$ as $m \rightarrow \infty$. The claim then follows from Proposition \ref{prop:q=2m} by letting $m\to\infty$.
\end{proof}

%%%%%%%% CONCLUSION %%%%%%%%%

\subsection{Conclusion of the proof}

We are now able to conclude the proof of Theorem \ref{thm:main-theorem}.

\begin{proof}[Proof of Theorem \ref{thm:main-theorem}]
In light of Proposition \ref{prop:lowerbound}, we only need to establish the upper bound.

It follows from  Corollary \ref{cor:p=1}, the trivial bound
\[
\norm{\widetilde{M}^{k}_{\rho,3\delta}f}_{L^\infty(Q_0)}\le \|f\|_{L^\infty(7 Q_0)}
\]
and the Riesz-Thorin Theorem (see e.g. \cite[Theorem 1.3.4]{GRFK}) that
\[
\norm{\widetilde{M}^{k}_{\rho,3\delta}f}_{L^p(Q_0)} \leq C_{n,k} \delta^{\frac{k-n}{2np}} \norm{f}_{L^p(7 Q_0)},
\]
for all $f \in L^p(7 Q_0)$,  $1 \le p \le \infty$.

 Since this bound is independent of $\rho\in\Gamma$, it follows from Lemma \ref{lem:sup-over-discretizations} that
 \begin{equation} \label{cotaencubos}
  \norm{M^{k}_{\delta}f}_{L^p(Q_0)}\le C_{n,k} \delta^{\frac{k-n}{2np}}\norm{f}_{L^p(7Q_0)}.
   \end{equation}

For each $z=(z_1,\ldots,z_n) \in \Z^n$ we denote
\[
Q_z = [z_1,z_1+1)\times\cdots\times [z_n,z_n+1).
\]
By translation invariance, \eqref{cotaencubos} continues to hold if we replace $Q_0$ by $Q_z$ on both sides of the inequality. This allows us to extend the bound to all of $\rr^n$. Given $f\in L^p(\rr^n)$, we have
\begin{align*}
\int_{\rr^n} |{M}^k_{\delta}f(x)|^p dx &= \displaystyle \sum_{z\in\Z^n} \int {\1}_{Q_z} |{M}^k_{\delta}f(x)|^p\,dx \\
&\le C_{n,k}^p \delta^{\frac{k-n}{2n}} \displaystyle \sum_{z\in\Z^n}   \int_{7Q_z} |f(x)|^p \,dx\\
&=  C_{n,k}^p \delta^{\frac{k-n}{2n}} \int \displaystyle \left(\sum_{z\in\Z^n}{\1_{7Q_z}}(x)\right) |f(x)|^p\, dx \\
&\leq C'  C_{n,k}^p \delta^{\frac{k-n}{2n}} \norm{f}^p_{L^p(\rr^n)},
\end{align*}
where $C'=\norm{\sum_{z\in\Z^n}{\1_{7Q_z}}}_\infty$.
\end{proof}

\section{An extension and an application}

\label{sec:extensions}

\subsection{An unrestricted extension}

In this section we extend Theorem \ref{thm:main-theorem} to the following unrestricted version:
\begin{defn} \label{def:unrestricted}
For $f \in L^1_{\textrm{loc}}(\rr^n)$ we define:

\begin{equation*}
\mathbf{M}^{k}_{\delta}f(x) = \sup\limits_{r: \rr^2\rightarrow (\delta,2]} A^k_{r,\delta}f(x),
\end{equation*}
where
\[
A^k_{r,\delta}f(x)=  \min\limits_{j=1}^{N}  \displaystyle\frac{1}{\mathcal{L}(S_{k,\delta}^{j}(x,r(x)))} \int _ {S_{k,\delta}^{j}(x,r(x))} |f(y)| dy.
\]
\end{defn}

\begin{thm} \label{thm:main-theorem-unrestrcited} Given $0\leq k <n$, $p \in [1,\infty)$, there exist a positive constant $C_{n,k}$ such that
\[
\norm{\mathbf{M}^k_{\delta}f}_{L^p(\rr^n)} \leq C_{n,k} \delta^{\frac{k-n}{2np}} \norm{f}_{L^p(\rr^n)}  \, \, \textup{for all} \, \, f \in L^p(\rr^n).
\]
\end{thm}

We begin by considering the case where the side lengths are between $2^t$ and $2^{t+1}$.
\begin{defn}
 Given $0<\delta<1$ and $f \in L^1_{loc}(\rr^n)$, define
\begin{equation*}
M^{k}_{\delta,t}f(x) = \sup\limits_{2^t \leq r \leq2^{t+1}} \min\limits_{j=1}^{N}  \displaystyle\frac{1}{\mathcal{L}(S_{k,\delta}^{j}(x,r))} \int _ {S_{k,\delta}^{j}(x,r)} |f(y)| dy,
\end{equation*}
where $ t \in \Z$ is fixed.
\end{defn}

\begin{lemma}\label{lem:norms-rescaled}
 \[
 \norm{M^k_{\delta,t}}_{L^p \rightarrow L^p} = \norm{M^k_{2^{-t}\delta}}_{L^p \rightarrow L^p}.
\]
\end{lemma}

\begin{proof}
Let $2^t \le r \le 2^{t+1}$. We use the usual shorthand $\fint_A f = \mathcal{L}_n(A)^{-1} \int_A f$.
 Let $f \in L^p(\rr^n)$ and let $g(\cdot)=f(2^t\cdot)$.  Changing variables one can check that
\[
\fint_{S^j_{\delta}(x,r)} |f(y)| dy =  \fint_{S^j_{{2^{-t}\delta}}(2^{-t}x,\tilde{r})} |f(2^{t}y)| \, dy.
\]
Therefore $M^{k}_{2^{-t}\delta} g(2^{-t}x)= M^{k}_{\delta,t} f(x)$ and hence
\[
\|M^{k}_{2^{-t}\delta} g\|_p/\|g\|_p = \|M^{k}_{\delta,t} f\|_p/\|f\|_p,
\]
giving the claim.
\end{proof}

\begin{rmk}
Lemma \ref{lem:norms-rescaled} together with Theorem \ref{thm:main-theorem} show that if we allow $k$-skeletons of cubes with arbitrarily large side lengths then the corresponding maximal operator cannot be bounded.

On the other hand, if we allow side lengths  smaller than $\delta$, then $S_{\delta,k}(x,r)$ becomes an $n$-cube with center $x$ and side length $\approx \delta$, so we are back to averaging over full cubes, similar to the classical Hardy-Littlewood maximal operator. For this reason in Definition \ref{def:unrestricted} we restrict ourselves to cubes of sides between $\delta$ and $2$.
 \end{rmk}

\begin{proof}[Proof of Theorem \ref{thm:main-theorem-unrestrcited}]
Fix a function $r: \rr^2 \rightarrow (\delta,2]$ and define the level sets

\[
\Omega_{t}:=\{ x \in \rr^n: 2^t \leq r(x) < 2^{t+1}\}.
\]
Letting $\log$ be the base $2$ logarithm, we have
\begin{align*}
\int_{\rr^n} (A^k_{r,\delta}f(x))^p \, dx &= \displaystyle \sum_{t=\lfloor \log(\delta)\rfloor }^0 \int_{\Omega_t} (A^k_{r,\delta}f(x))^p \, dx \\
\\&\leq \displaystyle \sum_{t=\lfloor \log \delta\rfloor }^0 \int_{\rr^n} (M^k_{\delta,t}f(x))^p \, dx
\\&\leq \displaystyle \sum_{t=\lfloor \log\delta\rfloor }^0 \norm{M^k_{2^{-t}\delta}f}^p_{p}
\\ &\leq \displaystyle C_{n,k}\delta^{\frac{k-n}{2n}} \norm{f}^{p}_{p}  \sum_{t=-\infty }^0    2^{t\tfrac{n-k}{2n}}\\
&= C_{n,k}\delta^{\frac{k-n}{2n}}\norm{f}^{p}_{p}.
\end{align*}
Since this holds for every function $r$, we obtain the desired result.
\end{proof}

\subsection{A geometric corollary}

The following corollary of Theorem \ref{thm:main-theorem-unrestrcited} recovers the special case of \cite[Theorem 1.1(2)]{Thornton17} in which $\bdim(S)=n$.
 \begin{cor} For any $0 \leq k < n$ and bounded sets $A, B \subseteq \rr^n$ such that $\bdim(A)=n$ and $B$ contains the $k$-skeleton of a $n$-cube around every point of  $A$,
 \[
 \lbdim(B) \geq  n - \frac{1}{2}+\frac{k}{2n}.
 \]
 \end{cor}

 \begin{proof} Since $B$ is bounded, by rescaling we may assume that all the cubes have side lengths $\le 2$. Fix a small $\delta>0$. Note that for each $x\in A_\delta$ there is $r\in [\delta,2]$ such that $S_{k,\delta}(x,r) \subset B_{2\delta}$. It follows that
 \[
 \mathbf{M}^{k}_{\delta}\mathbf{1}_{B_{2\delta}}(x) =1 \quad\text{for all }x\in A_\delta.
  \]
  Then, by Theorem \ref{thm:main-theorem-unrestrcited} applied with $p=1$,
  \[
  \mathcal{L}_n(A_\delta) \le C_{n,k} \delta^{\tfrac{k-n}{2n}} \mathcal{L}_n(B_{2\delta}).
  \]
  The corollary now follows from the definition of box dimension in terms of volumes of neighborhoods, see \cite[Proposition 3.2]{Falconer03}.
  \end{proof}

\begin{rmk}
The proof yields a lower bound for $\bdim(B)$ in terms of $\bdim(A)$ for any values of the latter, but these bounds are worse than the sharp ones obtained in \cite{KNS18, Thornton17} unless $\bdim(A)=n$.
\end{rmk}

%\bibliographystyle{plain}
%\bibliography{refe}

\end{document}